\documentclass[11pt,bezier]{article}
\usepackage{amsmath,amssymb,amsfonts,euscript,graphicx}

\newtheorem{preproof}{{\bf \indent Proof.}}

\newenvironment{proof}[1]{\begin{preproof}{\rm
               #1}\hfill{$\Box$}}{\end{preproof}}

\newtheorem{thm}{{\bf\indent Theorem}}[section]
\newtheorem{prop}{\bf\indent Proposition}[section]
\newtheorem{remark}{\bf\indent Remark}[section]
\newtheorem{lem}{\bf\indent Lemma}[section]

\title{\bf \large Coloring of  cozero-divisor graphs\\ of commutative von Neumann regular rings}

\author{{\normalsize  {\sc R. Nikandish${}^{\mathsf{a}}$, {\sc M.J. Nikmehr${}^{\mathsf{b}}$} and {\sc M. Bakhtyiari${}^{\mathsf{b}}$}  }
}\vspace{3mm}\\
{\footnotesize{${}^{\mathsf{a}}$\it Department of Basic Sciences, Jundi-Shapur University of Technology,}}\\
{\footnotesize{\rm P.O. BOX \rm{64615-334},
Dezful, Iran}}\\
{\footnotesize{ $\mathsf{r.nikandish@jsu.ac.ir}$}}\\
{\footnotesize{${}^{\mathsf{b}}$\it Faculty of Mathematics, K.N. Toosi
University of Technology, }}\\
{\footnotesize{\rm P.O. BOX \rm{16315-1618}, Tehran, Iran}}\\
{\footnotesize{ $\mathsf{nikmehr@kntu.ac.ir}$}}\quad\quad
{\footnotesize{$\mathsf{m.bakhtyiari55@gmail.com}$}}\\
{\footnotesize{$\mathsf{}$ }}}
\date{}

\begin{document}

\maketitle

\begin{abstract}
{\small Let $R$ be a commutative ring with non-zero identity. The cozero-divisor graph of $R$, denoted by $\Gamma^{\prime}(R)$, is a graph with vertices in $W^*(R)$, which is the set of all non-zero and non-unit elements of $R$, and two distinct vertices $a$ and $b$ in $W^*(R)$
are adjacent if and only if $a\not\in Rb$ and $b\not\in Ra$. In this paper, we show that the cozero-divisor graph of a von Neumann regular ring with finite clique number is not only weakly perfect but also perfect. Also, an explicit formula for the clique number is given.
   }
\end{abstract}
\begin{center}\section{Introduction}\end{center}
\par
The cozero-divisor graphs associated with commutative rings, as the dual notion of zero-divisor graphs, was first introduced by Afkhami and Khashyarmanesh in \cite{Afkhami1}, where they investigated some fundamental properties on the structure of this graph and the relation between cozero-divisor and zero-divisor graphs. Study of the complement of cozero-divisor graphs and characterization of commutative rings with forest, star, double-star or unicyclic cozero-divisor graphs were made bay the same authors in \cite{Afkhami2}. Planar, outerplanar and ring graph cozero-divisor graphs may be found in \cite{Afkhami3}. Akbari et al. gave further results on rings with forest cozero-divisor graphs and diameter of cozero-divisor graphs associated with $R[x]$ and $R[[x]]$ (see \cite{Akbari1}). The cozero-divisor graph has also been studied in several other papers (e.g., \cite{Afkhami4, Akbari2, Ansari, Kala}). In this paper, we deal with the coloring cozero-divisor graphs problem. Interested readers may find some methods in coloring of graphs associated with rings in \cite{aalipour, maimani}. First we recall some terminology and notation.
\par
Throughout this paper, all rings are assumed to be commutative with identity. We denote by $\mathrm{Max}(R)$, ${U}(R)$, $W(R)$ and $\mathrm{Nil}(R)$, the set of all maximal ideals of $R$, the set of all invertible elements of $R$, the set of all non-unit elements of $R$ and the set of all nilpotent elements of $R$, respectively. For a subset $T$ of a ring $R$ we let $T^*=T\setminus\{0\}$. The ring $R$ is said to be \textit{reduced} if it has no non-zero
nilpotent element. The ring $R$ is called \textit{von\, Neumann regular} if for every $r\in R$, there exists an $s\in R$ such that $r=r^2s$. The \textit{krull dimension of} $R$, denoted by ${\rm dim}(R)$,
is the supremum of the lengths of all chains of prime ideals.  For any undefined notation or terminology in ring theory, we refer the reader to \cite{ati}.
\par
 Let $G=(V,E)$ be a graph, where $V=V(G)$ is the set of vertices and $E=E(G)$ is the set of edges.  By $\overline{G}$, we mean the complement graph of $G$. We write $u-v$, to denote an edge with ends $u,v$.  If $ U \subseteq V(G)$, then  by $N(U)$ we mean the set of all neighbors of  $U$ in $G$. A graph $H=(V_0,E_0)$ is called a \textit{subgraph of} $G$ if $V_0\subseteq V$ and $E_0 \subseteq E$. Moreover, $H$ is called an \textit{induced subgraph by} $V_0$,  denoted by $G[V_0]$, if $V_0\subseteq V$ and $E_0=\{\{u,v\}\in E\, |\,u,v\in V_0\}$. Also $G$ is called a \textit{null graph} if it has no edge.
 A \textit{clique} of $ G $ is a maximal complete subgraph of $ G $ and the number of vertices in the largest clique of $ G $, denoted by $\omega(G)$, is called the \textit{clique number} of $ G $. For a graph $ G $, let $ \chi(G) $ denote the \textit{vertex chromatic number} of $ G $, i.e., the minimal number of colors which can be assigned to the vertices of $ G $ in such a way that every two adjacent vertices have different colors. A graph $G$ is said to be \textit{weakly perfect} if $\omega(G)=\chi(G)$. A \textit{perfect graph} $G$ is a graph in which  every induced subgraph is weakly perfect.
For any undefined notation or terminology in graph theory, we refer the reader to \cite {west}.
\par
 Let $R$ be a commutative ring with nonzero identity. \textit{The cozero-divisor graph of} $R$, denoted by $\Gamma^{\prime}(R)$, is a graph with the vertex set $W^*(R)$ and two distinct vertices $a$ and $b$ in $W^*(R)$
are adjacent if and only if $a\not\in Rb$ and $b\not\in Ra$. In this paper, it is shown that the cozero-divisor graph of a von Neumann regular ring with finite clique number is weakly perfect. Moreover,  an explicit formula for the clique number is given. Finally, we strengthen this result; Indeed it is proved that this graph is perfect.


{\section{Clique and Chromatic Number of $\Gamma^{\prime}(R)$}\vspace{-2mm}

Let $R$ be a von Neumann regular ring and $\omega(\Gamma^{\prime}(R))<\infty$. The main of this section is to show that $\omega(\Gamma^{\prime}(R))=\chi(\Gamma^{\prime}(R))={n\choose [n/2]}$, where $n=|\mathrm{Min}(R)|$. First, we need a series of lemma.

\begin{lem}\label{color}
 Let $R$ be a ring. If  $\omega(\Gamma^{\prime}(R))<\infty$, then $R$ is a Noetherian ring.
\end{lem}
\begin{proof}
 {It is enough to show that every ideal of $R$ is finitely generated.  Suppose to the contrary,  there exists an ideal $I$ of $R$ which is generate by the set $(x_i)_{i\in \Lambda}$, where $|\Lambda|=\infty$ and it is not  generate by the set $(x_i)_{i\in \Upsilon}$, where $\Upsilon=\Lambda \setminus \{i\}$, for every $i\in \Lambda$. Thus $x_i\not\in Rx_j$ and $x_j\not\in Rx_i$,  for every two distinct elements $i,j\in \Lambda$. Hence the set $(x_i)_{i\in \Lambda}$ is a clique of $\Gamma^{\prime}(R)$ and so $\omega(\Gamma^{\prime}(R))=\infty$, which is a contradiction. Therefore, every ideal of $R$ is finitely generated.
   }
\end{proof}

\begin{lem}\label{color1}
 Let $R$ be a von Neumann regular ring. If  $\omega(\Gamma^{\prime}(R))<\infty$, then $R\cong F_1\times\cdots \times F_n$, where every $F_i$ is a field and $|\mathrm{Min}(R)|=n$.
\end{lem}
\begin{proof}
 { By \cite[Theorem 3.1]{Huckaba}, $R$ is a reduced ring and ${\rm dim}(R)=0$. Moreover, by Lemma \ref{color}, $R$ is a Noetherian ring. Thus $R$ is a reduced Artinian ring. The result now follows from \cite[Theorem 8.7]{ati}.}
\end{proof}

\begin{lem}\label{lemma1d}
 Let $R$ be a ring. Then the following statements are equivalent.

 $(1)$ $a-b$ is an edge of $\Gamma^{\prime}(R)$.

 $(2)$ $Ra\nsubseteq Rb$ and $Rb\nsubseteq Ra$.
\end{lem}
\begin{proof}
{It is straightforward. }

\end{proof}
\begin{lem}\label{lemma neighbour}
 Let $R$ be a ring and $x,y\in V(\Gamma^{\prime}(R))$ such that $Ra=Rb$. Then $N(a)=N(b)$.
\end{lem}
\begin{proof}
 {Suppose that $c\in N(a)$. By Lemma \ref{lemma1d}, $Ra\nsubseteq Rc$ and $Rc\nsubseteq Ra$. Since $Ra=Rb$, we deduce that $Rb\nsubseteq Rc$ and $Rc\nsubseteq Rb$ and thus by Lemma \ref{lemma1d}, $c\in N(b)$. Hence  $N(a)\subseteq N(b)$. Similarly, $N(b)\subseteq N(a)$, as desired.
  }
\end{proof}

\begin{lem}\label{perfect13}
 Let $2 \leq n<\infty$ be an integer and
$R=\mathbb{Z}_{2}\times \cdots \times \mathbb{Z}_{2}$ {\rm ($n$ \textit{times})}.
  Then  $$\omega(\Gamma^{\prime}(R))=\chi(\Gamma^{\prime}(R))={n\choose [n/2]}.$$
\end{lem}
\begin{proof}
{ Let $x=(x_1,\dots,x_n)\in V(\Gamma^{\prime}(R))$. Obviously, $x_i=0$ for some $i\in \{1,\dots,n\}$.
Let $NZC(x)$ be the number zero $x_i$'s in $x$, for every $x=(x_1,\dots,x_n)\in V(\Gamma^{\prime}(R))$.
Clearly, $1\leq NZC(x)\leq n-1$, for every $x=(x_1,\dots,x_n)\in V(\Gamma^{\prime}(R))$. For every $1\leq i\leq n-1$, let
$$A_i=\{x=(x_1,\dots,x_n)\in V(\Gamma^{\prime}(R))|\,\,\, NZC(x)=i\}.$$

%
%

 It is easily seen that $V(\Gamma^{\prime}(R))=\cup_{i=1}^{n-1}A_i$ and $A_i\cap A_j=\varnothing$, for every $i\neq j$ and so $\{A_1,\dots,A_{n-1}\}$ is a partition of $V(\Gamma^{\prime}(R))$. We show that $\Gamma^{\prime}(R)[A_i]$ is a complete (induced) subgraph of $\Gamma^{\prime}(R)$, for every $1\leq i\leq n-1$. Let $x=(x_1\dots,x_n), y=(y_1,\dots,y_n)\in A_i$, for some $1\leq i\leq n-1$ and $x\neq y$. Since $ NZC(x)=NZC(y)$, there exist $1\leq i\neq j\leq n$ such that $x_i=0$, $y_i=1$ and  $x_j=1$, $y_j=0$. This implies that  $x\not\in Ry$ and $y\not\in Rx$ and so $x$ and $y$
are adjacent. Hence $\Gamma^{\prime}(R)[A_i]$ is a complete (induced) subgraph of $\Gamma^{\prime}(R)$, for every $1\leq i\leq n-1$.
 Furthermore, $|A_i|={n\choose i}$, for every $1\leq i \leq n$
and  $|A_t|\geq|A_i|$, for every $1\leq i \leq n-1$, where $t=[n/2]$. Let $i\neq j$ and $i<j<t$. Then
$|A_i|\leq|A_j|$ and
 for every  $x\in A_i$ there exists  a vertex $y\in A_j$ such that $Ry\subseteq Rx$. Thus by Lemma \ref{lemma1d}, $x$ is not adjacent to $y$ (by replacing one of the zero components of $y\in A_j$ by $1$, we have $x\in A_i$). Hence
$$\omega(\Gamma^{\prime}(R) [\cup_{i=1}^{t}A_i])=\chi(\Gamma^{\prime}(R)[\cup_{i=1}^{t}A_i])={n\choose t}.$$
 Similarly,
  $$\omega(\Gamma^{\prime}(R)[\cup_{i=t}^{n-1}A_i])=\chi(\Gamma^{\prime}(R)[\cup_{i=t}^{n-1}A_i])={n\choose t}.$$
Indeed, there are enough colors in $\Gamma^{\prime}(R)[A_t]$ to color $\Gamma^{\prime}(R)$.  Thus  $$\omega(\Gamma^{\prime}(R)=\chi(\Gamma^{\prime}(R))={n\choose t}.$$}\end{proof}

\begin{remark}\label{remark2}
{Let $G$ be a graph and $x\in V(G)$. If there exists a vertex $y\in V(G)$  which is not adjacent to $x$ and $N(x)=N(y)$, then  $\omega(G)=\omega(G\setminus \{x\})$ and  $\chi(G)=\chi(G\setminus \{x\})$.
}
\end{remark}

We are now in a position to state our main result of this section.
\begin{thm}\label{perfect}
 Let $R$ be a von Neumann regular ring and $|\mathrm{Min}(R)|=n$. If  $|\omega(\Gamma^{\prime}(R))|<\infty$, then  $$\omega(\Gamma^{\prime}(R))=\chi(\Gamma^{\prime}(R))={n\choose [n/2]}.$$
\end{thm}
\begin{proof}
 { By Lemma \ref{color1},
  $R\cong F_1\times \cdots  \times F_n$, where $F_i$ is a field, for every $1\leq i\leq n<\infty$. Let
 $$A=\{(x_1,\dots,x_n)\in V(\Gamma^{\prime}(R))|\,\,\ x_i\in \{0,1\}\,\,\mathrm{for\,\, every}\,\,1\leq i\leq n\}.$$
Consider the following claims:

\textbf{Claim 1.}
 $\omega(\Gamma^{\prime}(R)[A])=\omega(\Gamma^{\prime}(R))$ and  $\chi(\Gamma^{\prime}(R)[A])=\chi(\Gamma^{\prime}(R))$.

  Suppose that $x=(x_1,\dots,x_n)$ and $y=(y_1,\dots,y_n)$ are vertices of $\Gamma^{\prime}(R)$. Define the relation $\sim$ on $V(\Gamma^{\prime}(R))$ as follows:
$x \sim y$, whenever ``$x_i=0$ if and only if $y_i=0$'',
for every $1\leq i\leq n$. Obviously,  $\sim$ is an equivalence relation on $V(\Gamma^{\prime}(R))$. Thus
$V(\Gamma^{\prime}(R))=\cup_{i=1}^{2^n-2}[x]_i$, where $[x]_i$ is the equivalence class of $x_i$ (We note that the number of equivalence classes is $2^n-2$). Let $[x]$ be a equivalence class of $x$. Then
$|[x]\cap A|=1$ and so one may choose $a\in [x]\cap A$ and $b\in [x]\setminus \{a\}$. Since
  $Ra=Rb$, by Lemma \ref{lemma neighbour}, $N(a)=N(b)$. By Remark \ref{remark2}, $\omega(\Gamma^{\prime}(R))=\omega(\Gamma^{\prime}(R)\setminus \{b\})$ and  $\chi(\Gamma^{\prime}(R))=\chi(\Gamma^{\prime}(R)\setminus \{b\})$.
If we continue this procedure for $|V(\Gamma^{\prime}(R))\setminus A|$ times, then we get
 $\omega(\Gamma^{\prime}(R)[A])=\omega(\Gamma^{\prime}(R))$ and  $\chi(\Gamma^{\prime}(R)[A])=\chi(\Gamma^{\prime}(R))$.

 \textbf{Claim 2.} $\omega(\Gamma^{\prime}(R)[A])=\omega(\Gamma^{\prime}(S))$ and  $\chi(\Gamma^{\prime}(R)[A])=\chi(\Gamma^{\prime}(S))$, where
  $S=\mathbb{Z}_{2}\times \cdots \times \mathbb{Z}_{2}$ ($n$ times).

 Let $x=(x_1,\dots,x_n)\in S\setminus \{0,1\}$ and $y=(y_1,\dots,y_n)\in A$. Consider the map
 $\varphi : S\setminus \{0,1\}\longrightarrow A$ defined by the rule  $\varphi(x)=y$, whenever $x_i=0$ if and only if $y_i=0$.
 It is not hard to check that $\varphi$ is well-defined, bijective and if $x,y\in S\setminus \{0,1\}$ such that $x$ is adjacent $y$, then  $\varphi(x)$ is adjacent $\varphi(y)$. This implies that
 $\Gamma^{\prime}(S)\cong \Gamma^{\prime}(R)[A]$ and thus
  $\omega(\Gamma^{\prime}(R)[A])=\omega(\Gamma^{\prime}(S))$ and  $\chi(\Gamma^{\prime}(R)[A])=\chi(\Gamma^{\prime}(S))$.

 By Claims 1,2 and  Lemma \ref{perfect13},
   $$\omega(\Gamma^{\prime}(R))=\chi(\Gamma^{\prime}(R))=\omega(\Gamma^{\prime}(R)[A])=\chi(\Gamma^{\prime}(R)[A])=\omega(\Gamma^{\prime}(S))=\chi(\Gamma^{\prime}(S))={n\choose [n/2]}.$$}
\end{proof}

We close this section with the following proposition.
\begin{prop}\label{perfecte}
 Let $R$ be a ring which is not an integral domain. If  $|\omega(\Gamma^{\prime}(R))|<\infty$, then  $\Gamma^{\prime}(R)$ is a null graph if and only if $(R,\mathfrak{m})$ is local ring and $\mathfrak{m}$ is principal.
\end{prop}
\begin{proof}
 { First, suppose that $\Gamma^{\prime}(R)$ is a null graph. If $R$ is  not local, then one may choose  $x\in \mathfrak{m}_1\setminus\mathfrak{m}_2$ and $y\in \mathfrak{m}_2\setminus\mathfrak{m}_1$, where $\mathfrak{m}_1,\mathfrak{m}_2\in\mathrm{Max}(R) $. Since $x$ is not adjacent to $y$, we find a contradiction. Thus  $(R,\mathfrak{m})$ is local ring. Also, by a similar argument to the proof of Lemma \ref{color}, one may show that  $\mathfrak{m}$ is principal.

To prove the converse, suppose that $(R,\mathfrak{m})$ is a local ring and $\mathfrak{m}$ is principal. We show that dim$(R)=0$. It is enough to show that  $\mathfrak{m}\in\mathrm{Min}(R)$. Assume that  $\mathfrak{p}\subseteq\mathfrak{m}$, for some $\mathfrak{p}\in\mathrm{Min}(R)$. Since $R$ is not an integral domain,  $\mathfrak{p}\neq (0)$ and so we  may pick $0\neq a\in \mathfrak{p}$. Since $\mathfrak{m}$ is principal,  $\mathfrak{m}=Rx$, for some $x\in R$. If $x\in \mathfrak{p}$, then $\mathfrak{p}=\mathfrak{m}$ and thus dim$(R)=0$. So let $x\not\in \mathfrak{p}$.
Since $\mathfrak{p}\subseteq \mathfrak{m}$, $a=r_1x$ for some $r_1\in R$. Also $x\not\in \mathfrak{p}$ implies that  $r_1\in \mathfrak{p}$ and thus $r_1=r_2x$, for some $r_2\in R$. Hence $a=r_2x^2$ and so $a\in\mathfrak{m}^2 $. If we continue this procedure,
then $a\in\mathfrak{m}^n $, for every positive integer $n$. Therefore $a\in \cap_{n=1}^{\infty} \mathfrak{m}^n$. This, together with \cite[Corolary 10.19]{ati} imply that $a=0$, a contradiction. Hence $\mathfrak{p}=\mathfrak{m}$ and so dim$(R)=0$.  Since $R$ is Noetherian with dim$(R)=0$, $R$ is an Artinian local ring. Finally, by \cite[Proposition 8.8]{ati}, every ideal of $R$ is principal and hence  $\Gamma^{\prime}(R)$ is a null graph.
 }\end{proof}


{\section{Perfectness of $\Gamma^{\prime}(R)$}\vspace{-2mm}
Let $R$ be a von Neumann regular ring and  $\omega(\Gamma^{\prime}(R))<\infty$. In this section, we show that $\Gamma^{\prime}(R)$ is a  perfect graph. We begin with the following celebrate result.

%
%
%

\begin{lem}\label{Color1}
 {\rm(\cite{Diestel} The Strong Perfect Graph Theorem)}  A graph $G$ is perfect if and only if neither $G$ nor $\overline{G}$ contains an induced odd cycle of length at least 5.
\end{lem}

\begin{thm}\label{perfect14}
 Let
$R=\mathbb{Z}_{2}\times \cdots \times \mathbb{Z}_{2}$ {\rm ($n$ times)}.
  Then $\Gamma^{\prime}(R)$ is perfect.
\end{thm}
\begin{proof}
 { By Lemma \ref{Color1}, it is enough to prove the following claims.

\textbf{Claim 1.} $\Gamma^{\prime}(R)$ contains no induced odd cycle of length at least 5.
  Assume to the contrary,
 $$a_1-a_2-\cdots-a_n-a_1$$
  is an induced odd cycle of length at least 5 in $\Gamma^{\prime}(R)$.

  By Lemma  \ref{lemma1d}, either $Ra_1\subseteq Ra_3$ or $Ra_3\subseteq Ra_1$. We show that these two cases lead to contradictions. First assume that the case $Ra_1\subseteq Ra_3$ happens. We continue the proof by proving the following subclaims.

  \textbf{Subclaim 1.} $Ra_1\subseteq Ra_i$, for every $3\leq i\leq n-1$.

     Clearly, $Ra_1\subseteq Ra_3$.  By Lemma  \ref{lemma1d}, $Ra_1\subseteq Ra_4$ or $Ra_4\subseteq Ra_1$. If $Ra_4\subseteq Ra_1$, then $Ra_4\subseteq Ra_3$, a contradiction, by Lemma  \ref{lemma1d}. So $Ra_1\subseteq Ra_4$. Again,  by Lemma  \ref{lemma1d},
   $Ra_1\subseteq Ra_5$ or $Ra_5\subseteq Ra_1$. If $Ra_5\subseteq Ra_1$, then since $Ra_1\subseteq Ra_4$,   $Ra_5\subseteq Ra_4$, a contradiction. Thus  $Ra_1\subseteq Ra_5$. Similarly,  $Ra_1\subseteq Ra_i$, for every $6\leq i\leq n-1$.

\textbf{Subclaim 2.} $Ra_2\subseteq Ra_i$, for every $4\leq i\leq n$. Obviously,
   $Ra_1\subseteq Ra_4$, by the Subclaim 1. By Lemma  \ref{lemma1d}, $Ra_2\subseteq Ra_4$ or $Ra_4\subseteq Ra_2$. If $Ra_4\subseteq Ra_2$, then  $Ra_1\subseteq Ra_2$, a contradiction. So $Ra_2\subseteq Ra_4$. Next, we show that $Ra_2\subseteq Ra_5$. If $Ra_5\subseteq Ra_2$, then since $Ra_2\subseteq Ra_4$, we deduce that $Ra_5\subseteq Ra_4$, a contradiction. Therefore $Ra_2\subseteq Ra_5$. Similarly, $Ra_2\subseteq Ra_i$,  for every $6\leq i\leq n$.

  Now, using Subclaims  1 and 2, we show that $Ra_3\subseteq Ra_1$.
  By Lemma \ref{lemma1d},  $Ra_3\subseteq Ra_5$ or $Ra_5\subseteq Ra_3$. If $Ra_5\subseteq Ra_3$, then by Subclaim 2, $Ra_2\subseteq Ra_3$, a contradiction. Thus $Ra_3\subseteq Ra_5$. We show that $Ra_3\subseteq Ra_6$. If $Ra_6\subseteq Ra_3$, then by Subcase 2, $Ra_2\subseteq Ra_3$, a contradiction. So $Ra_3\subseteq Ra_6$. Similarly,  $Ra_3\subseteq Ra_i$, for every $7\leq i\leq n$.  Since $Ra_1\subseteq Ra_3$,  $Ra_1\subseteq Ra_i$,  for every $5\leq i\leq n$, i.e.,   $Ra_1\subseteq Ra_n$, a contradiction.
    Thus $Ra_3\subseteq Ra_1$ and this contradicts Subclaim 1. Therefore, $\Gamma^{\prime}(R)$ contains no induced odd cycle of length at least 5.

  \textbf{Claim 2.}
$\overline{\Gamma^{\prime}(R)}$ contains no induced odd cycle of length at least 5.
    Assume to the contrary,
 $$a_1-a_2-\cdots-a_n-a_1$$
  is an induced odd cycle of length at least 5 in $\overline{\Gamma^{\prime}(R)}$.
  By Lemma  \ref{lemma1d},  we may assume that $Ra_1\subseteq Ra_2$. If $Ra_2\subseteq Ra_3$, then  $Ra_1\subseteq Ra_3$, a contradiction. Thus

  $$Ra_1\subseteq Ra_2,$$
  $$Ra_3\subseteq Ra_2.$$
If $Ra_4\subseteq Ra_3$, then $Ra_4\subseteq Ra_2$, a contradiction. Hence $Ra_3\subseteq Ra_4$. If $Ra_4\subseteq Ra_5$, then   $Ra_3\subseteq Ra_4$ implies that $Ra_3\subseteq Ra_5$, a contradiction. Thus

    $$Ra_3\subseteq Ra_4,$$
  $$Ra_5\subseteq Ra_4.$$
  Since $n$ is odd, by  continuing this procedure, we find

    $$Ra_{n-2}\subseteq Ra_{n-1},$$
  $$Ra_n\subseteq Ra_{n+1}=Ra_1.$$
  This implies that $Ra_n\subseteq Ra_1$ and since $Ra_1\subseteq Ra_2$, $Ra_n\subseteq Ra_2$, a contradiction.
  Therefore, $\overline{\Gamma^{\prime}(R)}$ contains no induced odd cycle of length at least 5.

 The proof now is complete.}
\end{proof}

\begin{remark}\label{remark2}
{Let $G$ be a graph and $x\in V(G)$. If there exists a vertex $y\in V(G)$  which is not adjacent to $x$ and $N(x)=N(y)$, then  $G$ is perfect if and only if $G\setminus \{x\}$ is perfect.
}
\end{remark}
We close this paper with the following result.
  \begin{thm}\label{perfectw}
    Let $R$ be a von Neumann regular ring and  $\omega(\Gamma^{\prime}(R))<\infty$. Then $\Gamma^{\prime}(R)$ is a  perfect graph.
\end{thm}
\begin{proof}
 { Since $|\omega(\Gamma^{\prime}(R))|<\infty$, it follows from Lemma \ref{color1} that
  $R\cong F_1\times \cdots  \times F_n$, where $F_i$ is a field, for every $1\leq i\leq n<\infty$. Let
 $$A=\{(x_1,\dots,x_n)\in V(\Gamma^{\prime}(R))|\,\,\ x_i\in \{0,1\}\,\,\mathrm{for\,\, every}\,\,1\leq i\leq n\}.$$
 By Lemma \ref{lemma neighbour} and Remark \ref{remark2}, it is not hard to check that $\Gamma^{\prime}(R)$ is  perfect graph if and only if
  $\Gamma^{\prime}(R)[A]$ is perfect. In fact if
 $$a_1-a_2-\cdots-a_n-a_1$$
  is an induced odd cycle of length at least 5 in $\overline{\Gamma^{\prime}(R)}$ or $\Gamma^{\prime}(R)$, then $Ra_i\neq Ra_j$, for every
  $1\leq i,j\leq n$, $i\neq j$.
 By the proof of Theorem \ref{perfect}, we find that  $\Gamma^{\prime}(R)[A]\cong \Gamma^{\prime}(S)$, where $S=\mathbb{Z}_{2}\times \cdots \times \mathbb{Z}_{2}$ ($n$ times). Thus  $\Gamma^{\prime}(R)$ is perfect if and only if $\Gamma^{\prime}(S)$ is perfect. The result now follows from Lemma \ref{perfect14}.}
\end{proof}




{}


\begin{thebibliography}{}{\small

\bibitem{aalipour} G. Aalipour, S. Akbari, R. Nikandish, M.J. Nikmehr and F. Shaveisi. On the coloring
of the annihilating-ideal graph of a commutative ring, Discrete. Math. 312 (2012) 2620--2626.

\bibitem{Afkhami1} M. Afkhami, K. Khashyarmanesh, The cozero-divisor graph of a commutative ring,
Southeast Asian Bull. Math.  35 (2011) 753--762.

\bibitem{Afkhami2} M. Afkhami, K. Khashyarmanesh, On the cozero-divisor graphs of  commutative rings and their complements,
Bull. Malays. Math. Sci. Soc. 35 (2012) 935–-944.

\bibitem{Afkhami3} M. Afkhami, M. Farrokhi D. G., K. Khashyarmanesh, Planar, outerplanar and ring graph cozero-divisor graphs.
Ars Comb. 131 (2017) 397--406.

\bibitem{Afkhami4} M. Afkhami, K. Khashyarmanesh, On the cozero-divisor graphs of and comaximal graphs of commutative rings, J. Algebra Appl. 12 (2013) 1250173 [9 pages].

\bibitem{Akbari1} S. Akbari, F. Alizadeh, S. Khojasteh, Some results on Cozero-divisor graph of a commutative ring, J. Algebra Appl. 13 (2014)  1350113  [14 pages].

\bibitem{Akbari2} S. Akbari,  S. Khojasteh, Commutative rings whose cozero-divisor graphs unicyclic or of bounded degree, Comm. Algebra 42 (2013) 1594--1605.

\bibitem{Ansari} H. Ansari-Toroght, F. Farshadifar,  Sh.Habibi, The cozero-divisor graph relative to finitely generated modules, Miskolc Mathematical Notes
14 (2013) 749–-756.

\bibitem{ati} M. F. Atiyah, I.G. Macdonald, Introduction to
Commutative Algebra, Addison-Wesley Publishing Company, 1969.


\bibitem{Diestel} R. Diestel, Graph Theory, NY, USA: Springer-Verlag, 2000.


\bibitem{Huckaba} J. A. Huckaba, Commutative Rings With Zero Divisors, 2nd ed., Prentice Hall, Upper Saddle River (1988).

\bibitem{Kala} S. Kavitha, R. Kala, On the genus of graphs from commutative rings, AKCE International Journal of Graphs and Combinatorics, 14 (2017) 27--34.


\bibitem{maimani} H. R. Maimani, M. R. Pournaki, S. Yassemi, A class of weakly perfect graphs, Czech. Math. Journal 60 (2010) 1037-–1041.

\bibitem{west} D. B. West, Introduction to Graph Theory, 2nd ed., Prentice Hall, Upper Saddle River (2001).
}
\end{thebibliography}
\end{document}